\documentclass[11pt, reqno]{amsart}
\usepackage{a4wide}
\usepackage{amssymb}
\usepackage{tikz} 
\usepackage{bbm}
\usepackage
[colorlinks = true,
            linkcolor = blue,
            urlcolor  = blue,
            citecolor = blue,
            anchorcolor = blue]{hyperref}
\usepackage[initials]{amsrefs}

\newtheorem{thm}{Theorem}[section]
\newtheorem{lem}[thm]{Lemma}
\newtheorem{prop}[thm]{Proposition}
\newtheorem{obs}[thm]{Observation}
\newtheorem{rem}[thm]{Remark}

\begin{document}
\title{The Scaled Polarity transform and related inequalities}

\author[S. Gilboa]{Shoni Gilboa}
\address[S. Gilboa]{Department of Mathematics and Computer Science,\newline\indent
The Open University of Israel,  Ra’anana, Israel}

\author[A. Segal]{Alexander Segal}
\address[A. Segal]{Department of Mathematics,\newline\indent
Afeka Academic College of Engineering, Tel Aviv, Israel}
\email{segalale@gmail.com}

\author[B. A. Slomka]{Boaz A. Slomka}
\address[B. A. Slomka]{Department of Mathematics and Computer Science, \newline\indent The Open University of Israel, Ra’anana, Israel}
\email{slomka@openu.ac.il}

\begin{abstract}
In this paper we deal with generalizations of the Mahler volume product for log-concave functions. We show that the polarity transform $\mathcal A$ can be rescaled so that the Mahler product it induces has upper and lower bounds of the same asymptotics.  We discuss a similar result for the $\mathcal J$ transform.
As an application, we extend the K{\"o}nig-Milman duality of entropy result to the class of geometric log-concave functions. 
\end{abstract}

\maketitle

\section{Introduction}
\subsection{Background}
Given a convex body $K$, the Mahler volume of $K$ is given by
\[
\mathcal{P}(K) =
\inf
\left\{
{\rm Vol}_n(K - x) \cdot {\rm Vol}_n((K - x)^\circ) \,:\, x\in K
\right\},
\]
where $K^\circ$ denotes the dual body of $K$. That is,
\[
K^\circ =
\left\{
y\in{\mathbb R}^n \,:\, \forall x\in K,\langle x,y\rangle\le 1
\right\}.
\]
The classical Blaschke-Santal\'{o} inequality states that the Mahler
volume is maximized by ellipsoids.
Blaschke was the first to prove this fact in dimensions $2$ and $3$. Later, Santal\'{o} \cite{Sant}
extended this inequality to every dimension (see \cite{MeyPaj} for a simple proof
due to Meyer and Pajor). 
Finding the minimal value of the
Mahler volume remains an open problem to this day. Mahler conjectured
that for centrally symmetric bodies, the Mahler volume attains its minimum value on cubes, i.e. $\mathcal{P}(K)\ge \mathcal{P}(B_\infty^n) = \frac{4^n}{n!}$. In the general case, where the body $K$ is not assumed
to be centrally symmetric, minimizers are conjectured to be centered simplices. Mahler \cite{Mahler} solved the planar case, and the centrally symmetric three dimensional case has been settled by Iriyeh and Shibata \cite{IriShi}, see also \cite{FHMRZ} for a simplified proof.

The Mahler conjecture would imply that the minimum and
maximum of $\mathcal{P}(K)$ differ only by a factor of $c^n$ for some universal constant $c>0$ or, in other words, that $\mathcal{P}(K)^{\frac{1}{n}}$ is of the order $\frac{1}{n}$ for any convex body. The latter was verified in the Bourgain-Milman theorem \cite{BM}, where it was
proved that there exists a universal constant $c>0$ such that
\begin{equation*}\label{Ineq-Bour-Mil}
c \le \left( \frac{\mathcal{P}(K)}{\mathcal{P}(B_2^n)} \right) ^\frac1n,
\end{equation*}
thus settling the Mahler conjecture in the asymptotic sense. In
recent years, several new proofs of the Bourgain-Milman inequality
have been discovered, by Kuperberg \cite{Kup}, Nazarov \cite{Naz},
and by Giannopoulos, Paouris, and Vritsiou \cite{GPV}; see also the survey  \cite{fradelizi2023volume} and references therein.

Functional analogues of the Mahler volume and the classical  inequalities associated with it, have also been extensively studied. Perhaps the most prominent one pertains to the class of convex functions from ${\mathbb R}^n$ to ${\mathbb R}\cup\{\infty\}$, which we denote by ${\rm Cvx}({\mathbb R}^n)$. 
The analog of volume of a
convex function $\varphi$ is usually defined to be the integral
$\int_{{\mathbb R}^n} e^{-\varphi}$, and the dual of $\varphi$ is obtained by
applying the Legendre transform
\[
({\mathcal L}\varphi)(y) = \sup_{x\in{\mathbb R}^n} (\langle x,y\rangle-\varphi(x)).
\]
The following infimum over translations is a functional analogue of
the Mahler volume.
\[
\mathcal{P}_{\mathcal L}(\varphi) = \inf_{a \in {\mathbb R}^n}
\left\{\int_{{\mathbb R}^n} e^{-\varphi\circ\tau_a}\int_{{\mathbb R}^n} e^{-{\mathcal L} (\varphi\circ\tau_a)}
\right\},
\]
where $\tau_a$ is the translation $x\mapsto x+a$ of ${\mathbb R}^n$. It is known that 
\begin{equation}\label{eq-ML}
c_{\mathcal L}^n \le \mathcal{P}_{\mathcal L}(\varphi) \le \mathcal{P}_{\mathcal L}\left( \frac{|\cdot|^2}{2} \right) =
(2\pi)^n, 
\end{equation}
where $c_{\mathcal L}>0$ is some universal constant and $\lvert\cdot\rvert$ is the standard Euclidean norm in ${\mathbb R}^n$.
The upper bound for $\mathcal{P}_{\mathcal L}$ is due to Ball 
\cite{Ball2} for even functions, and to Artstein, Klartag, and Milman \cite{AKM} for any convex function.
The lower bound for $\mathcal{P}_{\mathcal L}$ was established by Klartag and Milman in \cite{KM} for even functions, and by Fradelizi and Meyer in
\cite{FradMey} for all functions.

Note that, as the Mahler volume of convex bodies, the quantity $\mathcal{P}_{\mathcal L}(\varphi)^{\frac{1}{n}}$ is of the same order for any convex function. This fact is very useful, as it allows one to use duality in proving various inequalities on convex bodies and functions, without losing optimality, in the isomorphic sense. Notable examples are K\"{o}nig and Milman's metric entropy estimates \cite{KonigMilman87} and the result of Nasz\'odi, Nazarov and Ryabogin on fine approximation of convex bodies by poltyopes \cite{NNR17}.

For a large family of convex
functions there is another choice of duality besides the Legendre
transform, namely the polarity transform $\mathcal A$, which first appeared
in \cite{Rockafellar}. 
Let ${\rm Cvx}_0({\mathbb R}^n)$ be the class of \emph{geometric convex functions}, i.e., lower semi-continuous convex functions from ${\mathbb R}^n$ to $[0,\infty]$
vanishing at the origin.
Artstein and Milman proved in \cite{AM-Hidden}
that the only dualities (i.e. order reversing
involutions) of ${\rm Cvx}_0({\mathbb R}^n)$ are the Legendre transform $\mathcal L$ and the polarity
transform $\mathcal A$, which may be defined by
\[
({\mathcal A}\varphi)(y) = \sup_{x\in{\mathbb R}^n} \frac{\langle x,y\rangle-1}{\varphi(x)}.
\]
One may similarly consider the Mahler volume with respect to $\mathcal A$,
\[
\mathcal{P}_{\mathcal A}(\varphi) = \int_{{\mathbb R}^n} e^{-\varphi}\int_{{\mathbb R}^n} e^{-{\mathcal A} \varphi},
\]
for $\varphi$ in the class of functions
$$
\overline{\text{Cvx}}_{0}\left({\mathbb R}^{n}\right):=\left\{\varphi\in{\rm Cvx}_0({\mathbb R}^n)\,:\,0<\int_{{\mathbb R}^n} e^{-\varphi}<\infty\right\}.
$$
It was shown in \cite{ArtSlom} for even functions, that
\begin{equation}\label{eq-MA-upper-lower}
c^n \left({\rm Vol}_n\left(B_2^n\right)\right)^2
\le \mathcal{P}_{\mathcal A}(\varphi) \le
\left({\rm Vol}_n\left(B_2^n\right) n! \right)^2 \left(1+\frac{C}{n}\right)
\end{equation}
for some universal constants $c,C>0$.
Notice that unlike the previous cases, the lower and upper bounds for $(\mathcal{P}_{\mathcal A})^{1/n}$ are of different orders.
The existence of maximizers and minimizers of $\mathcal{P}_{\mathcal A}$ was recently shown \cite{Li}, however their exact form remains an open problem. 
Notice that the bounds in \eqref{eq-MA-upper-lower} are essentially tight, 
as $\mathcal{P}_{\mathcal A}(\lvert \cdot \rvert) = (n!{\rm Vol}(B_2^n))^2$ and $\mathcal{P}_{\mathcal A}(1^{\infty}_{B_2^n})=({\rm Vol}(B_2^n))^2$ (see below the definition of  $1^{\infty}_K$).

The polarity transform $\mathcal A$ commutes with the Legendre transform $\mathcal L$ (see, e.g., \cite{AM-Hidden}).
The composition ${\mathcal J} = {\mathcal A} {\mathcal L} =
{\mathcal L} {\mathcal A}\,$ of the two order reversing involutions $\mathcal L$ and $\mathcal A$ is an
order preserving involution (see \cite{AM-Hidden} and \cite{AFM} for
more details on the $\mathcal J$ transform). Since $\mathcal J$ is order preserving,
the product $\int_{{\mathbb R}^n} e^{-\varphi} \int_{{\mathbb R}^n} e^{-{\mathcal J} \varphi}$ cannot be
bounded from above, or bounded away from zero. Thus, it makes sense
to consider the ratio
\[
\mathcal{R}_{\mathcal J}(\varphi) = \frac{\int_{{\mathbb R}^n} e^{-{\mathcal J} \varphi}}{\int_{{\mathbb R}^n} e^{-\varphi}}
\]
for any $\varphi\in\overline{\text{Cvx}}_{0}\left({\mathbb R}^{n}\right)$ (in which case ${\mathcal J}\varphi\in\overline{\text{Cvx}}_{0}\left({\mathbb R}^{n}\right)$ as well).
In \cite{FlorentinSegal} it was shown that there exists a universal constant $C>0$ such that for any $\varphi \in \overline{\text{Cvx}}_{0}\left({\mathbb R}^{n}\right)$,
\begin{equation} \label{eq:JMhalerUpperLower}
\frac{1}{\left(1+\frac{C}{n}\right)n!} \leq \mathcal{R}_{\mathcal J}(\varphi) \leq \left(1+\frac{C}{n}\right)n!
\end{equation}
and the upper and lower bounds are achieved by functions of the form
\[
\max\left\{1_K^\infty,\, a||\cdot||_{K}\right\},
\]
where $||\cdot||_{K}$ is the Minkowski functional of the body $K$,
and the convex indicator function $1_K^\infty$ is defined to be $0$
on $K$ and $+\infty$ otherwise. 
Notice, that once again, unlike the first two cases, the lower and upper bounds of $(\mathcal{R}_{\mathcal J})^{\frac{1}{n}}$ are of different orders of magnitude.

\subsection{Main results}

The aim of this note is to show, that the polarity transform $\mathcal A$ and the gauge transform $\mathcal J$ can be scaled so that both $(\mathcal{P}_{\mathcal A})^{\frac{1}{n}}$ and $(\mathcal{R}_{\mathcal J})^{\frac{1}{n}}$ have upper and lower bounds of the same magnitude. 
To this end, we introduce for any $\alpha\in(0,\infty)$ the scaled polarity transform 
\[
\varphi\mapsto{\mathcal A}_{\alpha}\varphi=\alpha{\mathcal A}\varphi
\]
and consider the Mahler volume with respect to ${\mathcal A}_{\alpha}$:
\[
\mathcal{P}_{{\mathcal A}_{\alpha}}(\varphi) = \int_{{\mathbb R}^n} e^{-\varphi}\int_{{\mathbb R}^n} e^{-{\mathcal A}_{\alpha} \varphi}.
\]

It is easy to check that this scaling of $\mathcal A$ does not affect its property of involution. However,  while it holds that ${\mathcal J} = {\mathcal A} {\mathcal L} = {\mathcal L} {\mathcal A}$, the transforms ${\mathcal A}_{\alpha}$ and $\mathcal L$ do not commute for $\alpha\neq 1$. Thus, their composition induces two order preserving transforms:

\[
{\mathcal J}_{\alpha}^l= {\mathcal A}_{\alpha}{\mathcal L}=\alpha {\mathcal J}\text{ and } {\mathcal J}_{\alpha}^r = {\mathcal L}{\mathcal A}_{\alpha}.
\]
Hence, one may consider the Santal\'o ratio for each:
\[
{\mathcal R}_{{\mathcal J}_{\alpha}^r}(\varphi) = \frac{\int_{{\mathbb R}^n} e^{-{\mathcal J}_{\alpha}^r\varphi}}{\int_{{\mathbb R}^n} e^{-\varphi}}, \qquad {\mathcal R}_{{\mathcal J}_{\alpha}^l}(\varphi) = \frac{\int_{{\mathbb R}^n} e^{-{\mathcal J}_{\alpha}^l\varphi}}{\int_{{\mathbb R}^n} e^{-\varphi}}.
\]
Note that for every $\varphi\in{\rm Cvx}_0({\mathbb R}^n)$, it holds that 
\begin{equation*}
({\mathcal J}_{\alpha}^r\varphi)(x)=\alpha({\mathcal J}\varphi)(x/\alpha)=({\mathcal J}_{\alpha}^l\varphi)(x/\alpha)
\end{equation*}
for every $x\in{\mathbb R}^n$, and so ${\mathcal R}_{{\mathcal J}_{\alpha}^r}=\alpha^n {\mathcal R}_{{\mathcal J}_{\alpha}^l}$. 
Thus, it suffices to  consider only one of these Santal\'o ratios.

Let ${\mathcal K}_0\left({\mathbb R}^n\right)$ denote the family of convex bodies (i.e., compact convex sets with non-empty interior) in ${\mathbb R}^n$ which contain the origin in their interior.
For every $\alpha\in(0,\infty)$ and every $K\in{\mathcal K}_0\left({\mathbb R}^n\right)$, it holds that
\begin{equation}
\label{eq:sl-norm-K}
{\mathcal R}_{{\mathcal J}_{\alpha}^l}\left(\lVert\cdot\rVert_K\right)=\frac{\int_{{\mathbb R}^n} e^{-\alpha 1^{\infty}_K}}{\int_{{\mathbb R}^n} e^{-\lVert\cdot\rVert_K}}=\frac{{\rm Vol}_n(K)}{\int_{{\mathbb R}^n} e^{-\lVert\cdot\rVert_K}}=\frac{1}{n!}.
\end{equation}

For every positive integer $n$, let $\rho_n$ be the unique real number in the interval $(0,\frac{1}{4})$ satisfying
$\left(1-\sqrt{1-4\rho_n}\right)\rho_n^{-\frac{1}{n+2}}e^{\sqrt{1-4\rho_n}}=\left(2\sqrt[n]{n!}/(n+2)\right)^{\frac{n}{n+2}}$
(see Lemma \ref{lem:rhon}).

\begin{thm}\label{thm:exact-Jl}
For every $\alpha\in[\rho_n (n+2)^2,\infty)$, the following statements hold.
\begin{enumerate}
\item
For every $\varphi \in \overline{\text{Cvx}}_{0}\left({\mathbb R}^{n}\right)$, we have ${\mathcal R}_{{\mathcal J}_{\alpha}^l}(\varphi)\leq\frac{1}{n!}$;
moreover, ${\mathcal R}_{{\mathcal J}_{\alpha}^l}(\varphi)=\frac{1}{n!}$ if and only if $\varphi=\lVert\cdot\rVert_K$ for some  $K\in{\mathcal K}_0\left({\mathbb R}^n\right)$.
\item
For every $\varphi \in \overline{\text{Cvx}}_{0}\left({\mathbb R}^{n}\right)$, we have ${\mathcal R}_{{\mathcal J}_{\alpha}^l}(\varphi)\geq\frac{n!}{\alpha^n}$;
moreover, ${\mathcal R}_{{\mathcal J}_{\alpha}^l}(\varphi)=\frac{n!}{\alpha^n}$ if and only if $\varphi={\mathcal J}\lVert\cdot\rVert_K=1^{\infty}_K$ for some $K\in{\mathcal K}_0\left({\mathbb R}^n\right)$.
\end{enumerate}
\end{thm}

For every $K\in{\mathcal K}_0\left({\mathbb R}^n\right)$,
$r\in[0,1]$ and $t_0\in(0,\infty)$, 
let $\psi_{K, r, t_0}\in\overline{\text{Cvx}}_{0}\left({\mathbb R}^{n}\right)$ be the function whose epi-graph (see Section \ref{sec:pre} for the definition of the epi-graph of a function) is the convex hull of the union of the epi-graph of $\|\cdot\|_{rK}$ and the epi-graph of $\max\left\{\|\cdot\|_{K}, 1^{\infty}_{t_0 K} \right\}$.
It holds that $\psi_{K, r, t_0}\leq\min\left\{\|\cdot\|_{rK}, \max\left\{\|\cdot\|_{K}, 1^{\infty}_{t_0 K} \right\}\right\}$ and, moreover, $\psi_{K, r, t_0}$ is the maximal convex function with this property, in the sense that $\psi_{K, r, t_0}\geq\varphi$ for every convex $\varphi:{\mathbb R}^n\to{\mathbb R}$ satisfying $\varphi\leq\min\left\{\|\cdot\|_{rK}, \max\left\{\|\cdot\|_{K}, 1^{\infty}_{t_0 K} \right\}\right\}$.
Explicitly,
\begin{subequations}
\label{eq:psi-def}
\begin{equation}
\label{eq:psi-def:a}
\psi_{K,0,t_0}=\max\left\{\|\cdot\|_{K}, 1^{\infty}_{t_0 K} \right\},
\end{equation}
and for $r\in(0,1]$, 
\begin{equation}
\label{eq:psi-def:b}
\psi_{K, r, t_0}=\max\left\{\|\cdot\|_{K},
\frac{1}{r}\|\cdot\|_{K}-\frac{1-r}{r}t_0\right\};
\end{equation}
\end{subequations}
in particular,
$\psi_{K,1,t_0}=\lVert\cdot\rVert_K$ (for every $t_0$).

\begin{thm}
\label{thm:tight-Jl}
There are universal constants $C,c>0$ such that for every $\alpha\in(1,\rho_n (n+2)^2)$ 
there are $r\in[0,1]$, $t_0\in(0,\infty)$ and 
$$1+\frac{c}{n}\leq\gamma\leq Cn\sqrt{n}\, e^{\frac{2\alpha}{n}}
$$
such that for every $K\in{\mathcal K}_0\left({\mathbb R}^n\right)$ and every $\varphi\in\overline{\text{Cvx}}_{0}\left({\mathbb R}^{n}\right)$,
$$
{\mathcal R}_{{\mathcal J}_{\alpha}^l}(\varphi)\leq {\mathcal R}_{{\mathcal J}_{\alpha}^l}\left(\psi_{K,r,t_0}\right)
=\frac{\gamma n!}{\alpha^n}$$
and 
$$
{\mathcal R}_{{\mathcal J}_{\alpha}^l}(\varphi)\geq{\mathcal R}_{{\mathcal J}_{\alpha}^l}\left({\mathcal J}_{\alpha}^l\psi_{K,r,t_0}\right)
=\frac{1}{\gamma n!}.$$
\end{thm}

For $\varphi\in\overline{\text{Cvx}}_{0}\left({\mathbb R}^{n}\right)$, denote the barycenter of $e^{-\varphi}$ with respect
to the Lebesgue measure by 
${\rm bar}\left(e^{-\varphi}\right):=\int_{{\mathbb R}^n} e^{-\varphi(x)}x\,{\rm d}x/\int_{{\mathbb R}^n} e^{-\varphi(x)}\,{\rm d}x$.

\begin{thm}
\label{thm:Lc-Santalo}
For every $\alpha\in[\rho_n(n+2)^2,\infty)$ and every $\varphi \in \overline{\text{Cvx}}_{0}\left({\mathbb R}^{n}\right)$ for which ${\rm bar}\left(e^{-\varphi}\right)=0$, we have
\begin{equation}
\label{eq:Lc-Santalo-large-alpha}
c_{\mathcal L}\frac{\sqrt[n]{n!}}{\alpha}\le
\left(\mathcal{P}_{{\mathcal A}_{\alpha}}(\varphi)\right)^{\frac{1}{n}} 
\le \frac{2\pi}{\sqrt[n]{n!}},
\end{equation}
where $c_{\mathcal L}$ is as in \eqref{eq-ML}.
Additionally, there is a universal constant $C>0$ such that for every $\alpha\in(1,\rho_n(n+2)^2)$ there is 
$
0\leq\delta\leq C\max\left\{\frac{\ln n}{n},\frac{\alpha}{n^2}\right\}
$
such that 
\begin{equation}
\label{eq:Lc-Santalo-small-alpha}
\frac{c_{\mathcal L}}{(1+\delta)\sqrt[n]{n!}}\le
\left(\mathcal{P}_{{\mathcal A}_{\alpha}}(\varphi)\right)^{\frac{1}{n}} 
\le \left(1+\delta\right)2\pi\frac{\sqrt[n]{n!}}{\alpha}
\end{equation}
for every $\varphi \in \overline{\text{Cvx}}_{0}\left({\mathbb R}^{n}\right)$ for which ${\rm bar}\left(e^{-\varphi}\right)=0$.

Moreover, for every $\alpha\in(0,\infty)$ and every even $\varphi\in\overline{\text{Cvx}}_{0}\left({\mathbb R}^{n}\right)$ it holds that
\begin{equation*}
\left(\mathcal{P}_{{\mathcal A}_{\alpha}}\left(\varphi\right)\right)^{\frac{1}{n}}\leq \sqrt[n]{n!}\left({\rm Vol}_n\left(B_{2}^{n}\right)\right)^{\frac{2}{n}}\left(\frac{1}{\sqrt[n]{n!}}+\frac{\sqrt[n]{n!}}{\alpha}\right)
\end{equation*}
(note that 
$$\sqrt[n]{n!}\left({\rm Vol}_n\left(B_2^n\right)\right)^{\frac{2}{n}}=2\pi-\Theta\left(\frac{\ln n}{n}\right),$$
i.e., 
$2\pi-C_0\frac{\ln n}{n}<\sqrt[n]{n!}\left({\rm Vol}_n\left(B_{2}^{n}\right)\right)^{\frac{2}{n}}<2\pi-c_0\frac{\ln n}{n}$
for some universal constants $C_0,c_0>0$).
\end{thm}

In particular, Theorem \ref{thm:Lc-Santalo} implies that if $\alpha$ is of the order of $n^2$, then
$\left(\mathcal{P}_{{\mathcal A}_{\alpha}}\left(\varphi\right)\right)^{1/n}$
is of order $1/n$ for every $\varphi \in \overline{\text{Cvx}}_{0}\left({\mathbb R}^{n}\right)$ for which ${\rm bar}\left(e^{-\varphi}\right)=0$.

\begin{rem}\label{rem:bry-assumption-Atheorem}
The lower bounds in Theorem \ref{thm:Lc-Santalo} are valid without assuming that ${\rm bar}\left(e^{-\varphi}\right)=0$ (see the proof in \S\ref{sec:A}). 
However, the requirement of ${\rm bar}\left(e^{-\varphi}\right)=0$ is necessary for the upper bounds in the theorem, since otherwise one may consider the indicator of a convex set $K$ with $0\in \partial K$ so that $\mathcal{P}_{{\mathcal A}_{\alpha}}(1_K^\infty)$ is unbounded.
\end{rem}

The upper bounds in Theorem \ref{thm:Lc-Santalo} are essentially tight and up to constants, the lower bounds are tight as well.
Indeed, for every $\alpha\in(0,\infty)$,
\begin{align*}
\left(\mathcal{P}_{{\mathcal A}_{\alpha}}\left(1^{\infty}_{B_2^n}\right)\right)^{\frac{1}{n}}&=\left(\int_{{\mathbb R}^n} e^{-1^{\infty}_{B_2^n}}\right)^{\frac{2}{n}}=\left({\rm Vol}_n\left(B_2^n\right)\right)^{\frac{2}{n}}
\\&=\sqrt[n]{n!}\left({\rm Vol}_n\left(B_2^n\right)\right)^{\frac{2}{n}}\frac{1}{\sqrt[n]{n!}}
=\left(2\pi-\Theta\left(\frac{\ln n}{n}\right)\right)\frac{1}{\sqrt[n]{n!}}
\end{align*}
and
\begin{align*}
\left(\mathcal{P}_{{\mathcal A}_{\alpha}}\left(\lvert\cdot\rvert\right)\right)^{\frac{1}{n}}&=\left(\int_{{\mathbb R}^n} e^{-\lvert\cdot\rvert}\cdot\int_{{\mathbb R}^n} e^{-\alpha\lvert\cdot\rvert}\right)^{\frac{1}{n}}=\frac{1}{\alpha}\left(\int_{{\mathbb R}^n} e^{-\lvert\cdot\rvert}\right)^{\frac{2}{n}}=\frac{1}{\alpha}\left(n!{\rm Vol}_n\left(B_2^n\right)\right)^{\frac{2}{n}}\\
&=\sqrt[n]{n!}\left({\rm Vol}_n\left(B_2^n\right)\right)^{\frac{2}{n}}\frac{\sqrt[n]{n!}}{\alpha}=\left(2\pi-\Theta\left(\frac{\ln n}{n}\right)\right)\frac{\sqrt[n]{n!}}{\alpha}.
\end{align*}

\subsection{An application}

As an application of Theorem \ref{thm:Lc-Santalo}, we provide a new functional generalization
of K\"{o}nig-Milman's duality of metric entropy result \cite{KonigMilman87}, which we recall next.

Let $K,T\subseteq{\mathbb R}^n$ be convex bodies. The covering number of $K$ by $T$ is given by 
$N(K,T)=\min\{N\in {\mathbb N}\,:\,\exists x_1,\dots,x_N\in{\mathbb R}^n,\,\,K\subseteq\bigcup_{i=1}^N(x_i+T)\}$. 
In \cite{KonigMilman87}, it was proven that if $K$ and $T$ are centrally symmetric then 
\begin{equation}\label{eq:KM_duality_bodies}
c^nN(T^\circ,K^\circ)\le N(K,T)\le C^nN(T^\circ,K^\circ)
\end{equation}
for some universal constants $c,C>0$. As shown in \cite{PajorMilman2000}, the same inequality holds under the weaker assumptions that $K$ (or $K^\circ$) and $T$ (or $T^\circ$) have their barycenters at the origin. 

Given measurable functions $f, g:{\mathbb R}^n\to[0,\infty)$,  the covering number $N(f,g)$ of $f$
by $g$ was defined in \cite{ArtSlom17-func_cov} as the infimum of $\mu({\mathbb R}^n)$
over all non-negative Borel measures $\mu$ on ${\mathbb R}^n$ satisfying $\mu*g\geq f$. Using both inequalities in \eqref{eq-ML}, a functional version of \eqref{eq:KM_duality_bodies} was given in \cite{ArtSlom17-func_cov} for geometric log-concave functions, where duality is induced by the Legendre transform. 

Using Theorem \ref{thm:Lc-Santalo}, we prove the following.

\begin{thm}
\label{thm:KM_polarity} Let $\varphi,\psi\in\overline{\text{Cvx}}_{0}\left({\mathbb R}^{n}\right)$ and fix some $\alpha$ of order $n^2$. Suppose that either ${\rm bar}\left(\varphi\right)=0$ or ${\rm bar}\left({\mathcal A}_\alpha\varphi\right)=0$,  and either ${\rm bar}\left(\psi\right)=0$ or ${\rm bar}\left({\mathcal A}_\alpha\psi\right)=0$. Then for some universal
constants $c,C>0$ we have 
\[
c^{n}N\left(e^{-{\mathcal A}_\alpha\psi},e^{-{\mathcal A}_\alpha\varphi}\right)\le N\left(e^{-\varphi}, e^{-\psi}\right)\le C^{n}N\left(e^{-{\mathcal A}_\alpha\psi},e^{-{\mathcal A}_\alpha\varphi}\right).
\]
\end{thm}

Note that by taking $f={\mathbbm 1}_{K},g={\mathbbm 1}_{T}$ Theorem \ref{thm:KM_polarity}
recovers König and Milman's result for convex bodies. Also note that
without the scaling of the polarity transform, Theorem \ref{thm:KM_polarity}
does not hold, e.g., for $f={\mathbbm 1}_{B_{2}^{n}}$ and $g=e^{-\text{\ensuremath{\left|\cdot\right|}}}$.
This can be verified using the volume bounds in \cite{ArtSlom17-func_cov}. 

The rest of the paper is organised as follows;
in \S\ref{sec:pre} we recall some properties of the polarity transform $\mathcal A$ and the gauge transform $\mathcal J$;
in \S\ref{sec:J} we prove Theorem \ref{thm:exact-Jl} and Theorem \ref{thm:tight-Jl};
in \S\ref{sec:A} we prove Theorem \ref{thm:Lc-Santalo};
finally, in \S\ref{sec:MP-KM} we prove Theorem \ref{thm:KM_polarity}.

\subsection*{Acknowledgment} 
We are grateful to the anonymous referee for their helpful comments.
The second and third named authors were funded by ISF grant no. 784/20.

\section{Preliminaries}
\label{sec:pre}
Let 
$$
{\rm epi}(\varphi) =
\left\{ (x, z)\in {\mathbb R}^n\times[0,\infty) \,:\, \varphi(x) \leq z \right\}$$ 
denote
the epi-graph of a function $\varphi:{\mathbb R}^n\to[0,\infty]$, 
and denote the sub-level sets of $\varphi$
by
\[
L_t\left(\varphi\right):=\left\{ x\in{\mathbb R}^{n}\,:\,\varphi\left(x\right)\leq t\right\}=\left\{ x\in{\mathbb R}^{n}\,:\,(x,t)\in{\rm epi}(\varphi)\right\} ,
\]
for any $0\le t<\infty$. 
Note that for every $\varphi\in{\rm Cvx}_0({\mathbb R}^n)$ and $t\in[0,\infty)$, the set 
$L_{t}\left(\varphi\right)$ is convex, by the convexity of $\varphi$, closed, by the lower semi-continuity of $\varphi$, and contains the origin; moreover, if $\varphi\in\overline{\text{Cvx}}_{0}\left({\mathbb R}^{n}\right)$ then $L_{t}\left(\varphi\right)$ is bounded for every $t\in[0,\infty)$ and there is a $t_0\in[0,\infty)$ such that for every $t>t_0$, the origin is an interior point of $L_{t}\left(\varphi\right)$, hence $L_{t}\left(\varphi\right)\in{\mathcal K}_0\left({\mathbb R}^n\right)$; additionally, for every $\varphi\in\overline{\text{Cvx}}_{0}\left({\mathbb R}^{n}\right)$ we have 
\begin{equation}
\label{eq:int-by-levels}
\int_{{\mathbb R}^n} e^{-\varphi}=\int_{0}^{\infty}e^{-t}{\rm Vol}_n\left(L_{t}(\varphi)\right)\,{\rm d}t.
\end{equation}

It was proved in \cite{ArtSlom} that
for any $\varphi\in{\rm Cvx}_0({\mathbb R}^n)$
and any $s,t>0$ it holds that
$L_{s}\left(\mathcal{A}\varphi\right)\subseteq\left(st+1\right)\left(L_{t}\left(\varphi\right)\right)^{\circ}$.
It follows, in terms of the scaled polarity,  that 
\begin{equation}
\label{eq:A-level-sets}
L_{s}\left({\mathcal A}_{\alpha}\varphi\right)\subseteq\left(\frac{st}{\alpha}+1\right)\left(L_{t}\left(\varphi\right)\right)^{\circ}.
\end{equation}

Consider the function
\mbox{$F:{\mathbb R}^n\times {\mathbb R}^+ \to {\mathbb R}^n \times {\mathbb R}^+$} given by
\begin{equation}
\label{eq:Fdef}
F(x,z) = \left( \frac{x}{z}, \frac{1}{z} \right).
\end{equation}
The map $F$ is an involution that induces the $\mathcal J$ transform (see
\cite{AM-Hidden}), in the following sense. 
For any $\varphi \in {\rm Cvx}_0({\mathbb R}^n)$
we have:
\begin{equation}
\label{eq:JFepi}
{\rm epi}({\mathcal J}\varphi)=F({\rm epi}(\varphi)),
\end{equation}

Consequently, $
L_t({\mathcal J}_{\alpha}^l \varphi)=L_{t/\alpha}({\mathcal J}\varphi)=\frac{t}{\alpha}L_{\alpha/t}(\varphi)$ for every $t\in(0,\infty)$
and it follows, by \eqref{eq:int-by-levels}, that
\begin{align}
\int_{{\mathbb R}^n} e^{-{\mathcal J}_{\alpha}^l \varphi} &=
\int_{0}^{\infty}e^{-t}{\rm Vol}_n\left(L_{t}({\mathcal J}_{\alpha}^l\varphi)\right)\,{\rm d}t=
\int_{0}^{\infty}e^{-t}(t/\alpha)^n{\rm Vol}_n\left(L_{\alpha/t}(\varphi)\right)\,{\rm d}t\nonumber\\
&=\int_0^\infty \alpha e^{-\frac\alpha z}z^{-(n+2)}{\rm Vol}_n\left(L_z(\varphi)\right)\,{\rm d}z.\label{eq:Jint-by-levels}
\end{align}

\begin{obs} 
\label{obs:J-UpperLowerBound}
For every $\alpha\in(0,\infty)$ and every $\varphi\in\overline{\text{Cvx}}_{0}\left({\mathbb R}^{n}\right)$, 
it holds that 
\begin{equation*}
{\mathcal R}_{{\mathcal J}_{\alpha}^l}({\mathcal J}_{\alpha}^l\varphi)\,{\mathcal R}_{{\mathcal J}_{\alpha}^l}(\varphi)=\frac{1}{\alpha^n}.
\end{equation*}
Consequently, 
$$
\inf_{\varphi\in\overline{\text{Cvx}}_{0}\left({\mathbb R}^{n}\right)}{\mathcal R}_{{\mathcal J}_{\alpha}^l}(\varphi)\sup_{\varphi\in\overline{\text{Cvx}}_{0}\left({\mathbb R}^{n}\right)}{\mathcal R}_{{\mathcal J}_{\alpha}^l}(\varphi)=\frac{1}{\alpha^n}
$$
for every $\alpha\in(0,\infty)$.
\end{obs}
\begin{proof}
For every $\varphi\in{\rm Cvx}_0({\mathbb R}^n)$, it holds that
\begin{equation}
\label{eq:Jscaling}    
\left({\mathcal J}\alpha\varphi\right)(x)=\frac{1}{\alpha}\left({\mathcal J}\varphi\right)\left(\alpha x\right)
\end{equation}
for every $x\in{\mathbb R}^n$.
Therefore, for every $\varphi\in{\rm Cvx}_0({\mathbb R}^n)$,
$$
({\mathcal J}_{\alpha}^l({\mathcal J}_{\alpha}^l\varphi))(x)=\alpha\left({\mathcal J}\alpha{\mathcal J}\varphi\right)(x)=\alpha\frac{1}{\alpha}\left({\mathcal J}{\mathcal J}\varphi\right)\left(\alpha x\right)=\varphi\left(\alpha x\right)
$$ 
for every $x\in{\mathbb R}^n$, and hence, for every $\varphi\in\overline{\text{Cvx}}_{0}\left({\mathbb R}^{n}\right)$, 
\begin{equation*}
{\mathcal R}_{{\mathcal J}_{\alpha}^l}({\mathcal J}_{\alpha}^l\varphi)\,{\mathcal R}_{{\mathcal J}_{\alpha}^l}(\varphi)=\frac{\int_{{\mathbb R}^n} e^{-{\mathcal J}_{\alpha}^l({\mathcal J}_{\alpha}^l\varphi)}}{\int_{{\mathbb R}^n} e^{-{\mathcal J}_{\alpha}^l\varphi}}\cdot\frac{\int_{{\mathbb R}^n} e^{-{\mathcal J}_{\alpha}^l\varphi}}{\int_{{\mathbb R}^n} e^{-\varphi}}=\frac{\int_{{\mathbb R}^n} e^{-{\mathcal J}_{\alpha}^l({\mathcal J}_{\alpha}^l\varphi)}}{\int_{{\mathbb R}^n} e^{-\varphi}}=\frac{1}{\alpha^n}.
\end{equation*}

Consequently, since $\{{\mathcal J}_{\alpha}^l\varphi:\varphi\in\overline{\text{Cvx}}_{0}\left({\mathbb R}^{n}\right)\}=\overline{\text{Cvx}}_{0}\left({\mathbb R}^{n}\right)$ (as $\mathcal J$ is an involution of $\overline{\text{Cvx}}_{0}\left({\mathbb R}^{n}\right)$), it holds that
\begin{align*}
\inf_{\varphi\in\overline{\text{Cvx}}_{0}\left({\mathbb R}^{n}\right)}{\mathcal R}_{{\mathcal J}_{\alpha}^l}(\varphi)&=\inf_{\varphi\in\overline{\text{Cvx}}_{0}\left({\mathbb R}^{n}\right)}{\mathcal R}_{{\mathcal J}_{\alpha}^l}({\mathcal J}_{\alpha}^l\varphi)=\inf_{\varphi\in\overline{\text{Cvx}}_{0}\left({\mathbb R}^{n}\right)}\frac{1}{\alpha^n {\mathcal R}_{{\mathcal J}_{\alpha}^l}(\varphi)}\\&=\frac{1}{\alpha^n\sup_{\varphi\in\overline{\text{Cvx}}_{0}\left({\mathbb R}^{n}\right)}{\mathcal R}_{{\mathcal J}_{\alpha}^l}(\varphi)},
\end{align*}
i.e.,
\begin{equation*}
\inf_{\varphi\in\overline{\text{Cvx}}_{0}\left({\mathbb R}^{n}\right)}{\mathcal R}_{{\mathcal J}_{\alpha}^l}(\varphi)\sup_{\varphi\in\overline{\text{Cvx}}_{0}\left({\mathbb R}^{n}\right)}{\mathcal R}_{{\mathcal J}_{\alpha}^l}(\varphi)=\frac{1}{\alpha^n}.
\qedhere
\end{equation*} 
\end{proof}

\section{Scaling the $\mathcal J$ transform}
\label{sec:J}

We basically follow the arguments presented in \cite{FlorentinSegal}.

For $\varphi\in\overline{\text{Cvx}}_{0}\left({\mathbb R}^{n}\right)$, define the function $\varphi^*:[0,\infty)\to[0,\infty)$, as follows; for every $z\in[0,\infty)$, 
$$\varphi^*(z):= \sqrt[n]{{\rm Vol}_n\left(L_z(\varphi)\right)}.$$
Note that the function $\varphi^*$ is increasing; moreover, $\varphi^*$ is concave, by the Brunn-Minkowski inequality, since for every $z_0,z_1\in[0,\infty)$
and $0\leq\lambda\leq 1$ it holds that
\begin{equation}
\label{eq:concavity-levels}
(1-\lambda)L_{z_0}(\varphi)+\lambda L_{z_1}(\varphi)\subseteq L_{(1-\lambda)z_0+\lambda z_1}(\varphi).
\end{equation}

For every $\alpha,\lambda\in(0,\infty)$ and $\varphi\in\overline{\text{Cvx}}_{0}\left({\mathbb R}^{n}\right)$, it holds, by \eqref{eq:Jint-by-levels} and \eqref{eq:int-by-levels}, that
\begin{align}
\left({\mathcal R}_{{\mathcal J}_{\alpha}^l}(\varphi)-\lambda\right)\int_{{\mathbb R}^n} e^{-\varphi}&=\int_{{\mathbb R}^n} e^{-{\mathcal J}_{\alpha}^l \varphi} -\lambda\int_{{\mathbb R}^n} e^{- \varphi}\nonumber\\
&=\int_0^\infty e^{-z}\left(h_{\alpha}(z)-\lambda\right)\left(\varphi^*(z)\right)^n\,{\rm d}z, \label{eq:sl-to-h}
\end{align}
where $h_{\alpha}:(0,\infty)\to{\mathbb R}$ is defined as follows; for every $z\in(0,\infty)$,
\begin{equation*}
h_{\alpha}(z):=\alpha e^{z-\frac{\alpha}{z}}/z^{n+2}.
\end{equation*}

\begin{lem}\label{lem:sign-changes-general}
For every $\alpha,\lambda\in(0,\infty)$, 
either of the following holds.
\begin{enumerate}
\item 
The function $h_{\alpha}-\lambda$ has at most two positive real roots. Additionally, there is $z_0\in(0,\infty)$ such that $h_{\alpha}(z)\leq\lambda$ for $z\in(0,z_0)$, and $h_{\alpha}(z)\geq\lambda$ for $z\in(z_0,\infty)$.

\item The function $h_{\alpha}-\lambda$ has three roots $0<z_1<z_2<z_3<\infty$, and it holds that $h_{\alpha}(z)<\lambda$ for $z\in(0,z_1)\cup(z_2,z_3)$ and $h_{\alpha}(z)>\lambda$ for $z\in(z_1,z_2)\cup(z_3,\infty)$.
\end{enumerate}
\end{lem}

\begin{proof}
Note that for every $z\in(0,\infty)$,
\begin{equation}
\label{eq:derivative-h}
h_{\alpha}'(z) = \alpha \frac{e^{z-\frac{\alpha}{z}}}{z^{n+2}}\left(1 + \frac{\alpha}{z^2}- \frac{n+2}{z}\right).
\end{equation}
Hence, if $\alpha\geq (n+2)^2/4$ then $h_{\alpha}'$ is non-negative in $(0,\infty)$ and has at most one positive real root; if $0<\alpha<(n+2)^2/4$, then $h_{\alpha}'$ has two roots $0<\zeta_1(\alpha)<\zeta_2(\alpha)$.

It follows that if $\alpha\geq (n+2)^2/4$ then $h_{\alpha}$ is strictly increasing in $(0,\infty)$ and if $0<\alpha<(n+2)^2/4$ then $h_{\alpha}$ is strictly increasing in the interval $(0,\zeta_1(\alpha)]$, strictly decreasing in the interval $[\zeta_1(\alpha),\zeta_2(\alpha)]$ and strictly increasing in the interval $[\zeta_2(\alpha),\infty)$.

Since $h_{\alpha}(z)\xrightarrow[z\to 0]{}0$ and $h_{\alpha,\lambda}(z)\xrightarrow[z\to \infty]{}\infty$,
it follows that the second alternative holds if $0<\alpha<\frac{1}{4}(n+2)^2$, $h_{\alpha}(\zeta_1(\alpha))>\lambda$ and $h_{\alpha}(\zeta_2(\alpha))<\lambda$, and the first alternative holds otherwise.
\end{proof}

For every $K\in{\mathcal K}_0\left({\mathbb R}^n\right)$, $r\in[0,1]$ and $t_0\in(0,\infty)$, recall the definition of $\psi_{K,r,t_0}\in\overline{\text{Cvx}}_{0}\left({\mathbb R}^{n}\right)$ from \eqref{eq:psi-def}, and note that
\begin{equation}
\label{eq:psi-p}
\psi_{K, r, t_0}^*
=\sqrt[n]{{\rm Vol}_n(K)}\,p_{r,t_0},
\end{equation}
where the function $p_{r,t_0}:[0,\infty)\to[0,\infty]$ is defined as follows; for every $t\in[0,\infty),$
$$p_{r,t_0}(t)=\begin{cases}
t & 0\leq t\leq t_0,\\
rt+(1-r)t_0 & t\geq t_0.   
\end{cases}$$

The following lemma plays a pivotal role in the proofs of Theorem \ref{thm:exact-Jl} and Theorem \ref{thm:tight-Jl}.

\begin{lem}
\label{lem:piecewise}
Let $f:(0,\infty)\to{\mathbb R}$ and let $\psi:[0,\infty)\to[0,\infty)$ be a concave increasing function which is not identically zero.
\begin{enumerate}
\item 
Suppose that $f$ has at most two positive real roots and that there is $z_0\in(0,\infty)$ such that $f(z)\leq 0$ for $z\in(0,z_0)$ and $f(z)\geq 0$ for $z\in(z_0,\infty)$. Then  
$$\int_0^{\infty} f(z)\left(\psi(z)\right)^n \,{\rm d}z\leq \int_0^{\infty}f(z)\left(\frac{\psi(z_0)}{z_0}z\right)^n \,{\rm d}z$$ and equality occurs if and only if $\psi$ is linear, i.e., $\psi(z)=\frac{\psi(z_0)}{z_0}z$ for every $z\in(0,\infty)$.
\item 
Suppose that there are $0<z_1<z_2<z_3<\infty$ such that $f(z)<0$ for $z\in(0,z_1)\cup(z_2,z_3)$ and $f(z)>0$ for $z\in(z_1,z_2)\cup(z_3,\infty)$. 
Then there are $r\in[0,1]$ and $t_0\in[z_1,z_2]$ such that
$$
\int_0^{\infty} f(z)\left(\psi(z)\right)^n  \,{\rm d}z
\leq \int_0^\infty f(z)\left(\frac{\psi(z_1)}{z_1}
p_{r,t_0}(z)\right)^n \,{\rm d}z
$$
and equality occurs only if $\psi=\frac{\psi(z_1)}{z_1}
p_{r,t_0}$.
\end{enumerate}
\end{lem}

\begin{proof}
We first prove the first part.
Suppose that $f$ has at most two positive real roots, $f(z)\leq 0$ for $z\in(0,z_0)$, and $f(z)\geq 0$ for $z\in(z_0,\infty)$.
Let $l:[0,\infty)\to[0,\infty)$ be the linear function defined by $l(z)=\frac{\psi(z_0)}{z_0}z$ for every $z\in[0,\infty)$. 
Since $l(0)=0\leq\psi(0)$ and $l(z_0)=\psi(z_0)$, the linearity of $l$ and the concavity of $\psi$ imply that $\psi(z)\geq l(z)$ for $0<z<z_0$ and $\psi(z)\leq l(z)$ for $z>z_0$.
It follows that
$\int_0^\infty f(z)\left(\psi(z)\right)^n \,{\rm d}z\leq \int_0^\infty 
f(z)\left(l(z)\right)^n \,{\rm d}z$ and equality occurs if and only if $\psi=l$, as claimed.

The proof of the second part is similar, but slightly more elaborate.
Suppose that $f(z)<0$ for $z\in(0,z_1)\cup(z_2,z_3)$ and $f(z)>0$ for $z\in(z_1,z_2)\cup(z_3,\infty)$.
Let $a:=\frac{\psi(z_1)}{z_1}$ be the slope of the line  $l_1$ through the points $(0,0)$ and $(z_1,\psi(z_1))$, and let  $b:=\frac{\psi(z_3)-\psi(z_2)}{z_3-z_2}$ be the slope of the line $l_2$ through the points $(z_2,\psi(z_2))$ and $(z_3,\psi(z_3))$ (see Figure \ref{fig:1}). 
The concavity of $\psi$ implies that $a>0$ (otherwise, $\psi$ would be identically zero) and $b\leq a$; 
moreover, if $b<a$ and $(t_0,at_0)$ is the intersection point of the lines $l_1,l_2$, then $z_1\leq t_0\leq z_2$; if $b=a$, we choose $z_1\leq t_0\leq z_2$ arbitrarily.   
The concavity of $\psi$ further implies that $\psi(z)\geq az$ if  $z\in[0,z_1]$ and $\psi(z)\leq az$ if  $z\in[z_1,\infty)$, and that
$\psi(z)\geq b(z-t_0)+at_0$ if  $z\in[z_2,z_3]$ and $\psi(z)\leq b(z-t_0)+at_0$ if  $z\in[0,\infty)\setminus[z_2,z_3]$.
Hence, $\psi(z)\geq ap_{\frac{b}{a},t_0}(z)$ for $z\in[0,z_1]\cup[z_2,z_3]$ and $\psi(z)\leq ap_{\frac{b}{a},t_0}(z)$ for $z\in[z_1,z_2]\cup[z_3,\infty)$.
It follows that
$\int_0^{\infty} f(z)\left(\psi(z)\right)^n \,{\rm d}z\leq\int_0^\infty f(z)\left(a p_{\frac{b}{a},t_0}(z)\right)^n \,{\rm d}z$
and equality occurs if and only if $\psi=a p_{\frac{b}{a},t_0}$, as claimed.
\end{proof}

\begin{figure}[ht]
\centering
\begin{tikzpicture}
[scale=0.8, domain=0:10]
\draw[->] (0,0) -- (11,0); 
\draw[->] (0,0) -- (0,8); 
\draw[thick, smooth, variable=\x] 
plot (\x, {2*ln(2*\x+1.5}); 
\node at (10.2,6.2) {$\psi$};
\draw[thick, purple] (0,0) -- (1,3);
\draw[thick, purple] (1,3) -- (10,7);
\node[purple] at (10.4,7.1) {$a p_{\frac{b}{a},t_0}$};
\draw[dashed] (0.7,0) -- (0.7,2);
\node at (0.7,-0.3) {$z_1$};
\draw[thick, dotted, blue] (1,0) -- (1,3);
\node[blue] at (1.15,-0.3) {$t_0$};
\draw[dashed] (2.1,0) -- (2.1,3.45);
\node at (2.2,-0.3) {$z_2$};
\draw[dashed] (5.9,0) -- (5.9,5.25);
\node at (6,-0.3) {$z_3$};
\end{tikzpicture}
\caption{} 
\label{fig:1}
\end{figure}

For every $\alpha\in(0,\infty)$ denote, for convenience,
$$
\lambda_n(\alpha):=\sup_{\varphi\in\overline{\text{Cvx}}_{0}\left({\mathbb R}^{n}\right)}{\mathcal R}_{{\mathcal J}_{\alpha}^l}(\varphi)
$$

\begin{prop}\label{prop:1root}
For every $\alpha\in(0,\infty)$, it holds that $\lambda_n(\alpha)\leq\lambda$ for every $\lambda\geq\frac{1}{n!}$ such that the function $h_{\alpha}-\lambda$ has at most two positive real roots;
moreover, if the function $h_{\alpha}-\frac{1}{n!}$ has at most two roots, then $\lambda_n(\alpha)=\frac{1}{n!}$ and furthermore, ${\mathcal R}_{{\mathcal J}_{\alpha}^l}(\varphi)=\frac{1}{n!}$ for $\varphi\in\overline{\text{Cvx}}_{0}\left({\mathbb R}^{n}\right)$ if and only if $\varphi=\lVert\cdot\rVert_K$ for some $K\in{\mathcal K}_0\left({\mathbb R}^n\right)$.
\end{prop}

\begin{proof}
Let $\alpha\in(0,\infty)$ and $\lambda\geq\frac{1}{n!}$ such that the function $h_{\alpha}-\lambda$ has at most two roots.
By Lemma \ref{lem:sign-changes-general}, there is a $z_0\in(0,\infty)$ such that $h_{\alpha}(z)\leq\lambda$ for $z\in(0,z_0)$, and $h_{\alpha}(z)\geq\lambda$ for $z\in(z_0,\infty)$.
Let $\varphi\in\overline{\text{Cvx}}_{0}\left({\mathbb R}^{n}\right)$, and take some $K\in{\mathcal K}_0\left({\mathbb R}^n\right)$ such that $\sqrt[n]{{\rm Vol}_n(K)}=\varphi^*(z_0)/z_0$.
By \eqref{eq:sl-to-h} and the first part of Lemma \ref{lem:piecewise},
\begin{align*}
\left({\mathcal R}_{{\mathcal J}_{\alpha}^l}(\varphi)-\lambda\right)\int_{{\mathbb R}^n} e^{-\varphi}&=
\int_0^\infty e^{-z}\left(h_{\alpha}(z)-\lambda\right) \left(\varphi^*(z)\right)^n \,{\rm d}z\\
&\leq \int_0^\infty e^{-z}\left(h_{\alpha}(z)-\lambda\right)\left(\frac{\varphi^*(z_0)}{z_0}z\right)^n \,{\rm d}z,
\end{align*}
where equality occurs if and only if $\varphi^*$ is linear in $(0,\infty)$.
By \eqref{eq:sl-to-h} and \eqref{eq:sl-norm-K},
\begin{align*}
\int_0^\infty e^{-z}\left(h_{\alpha}(z)-\lambda\right)\left(\frac{\varphi^*(z_0)}{z_0}z\right)^n \,{\rm d}z&=\int_0^\infty e^{-z}\left(h_{\alpha}(z)-\lambda\right)\left(\sqrt[n]{{\rm Vol}_n(K)}\,z\right)^n \,{\rm d}z\\
&=\int_0^\infty e^{-z}\left(h_{\alpha}(z)-\lambda\right)\left(\lVert\cdot\rVert_K^*(z)\right)^n \,{\rm d}z\\
&=\left({\mathcal R}_{{\mathcal J}_{\alpha}^l}(\lVert\cdot\rVert_K)-\lambda\right)\int_{{\mathbb R}^n} e^{-\lVert\cdot\rVert_K}=\left(\frac{1}{n!}-\lambda\right)\int_{{\mathbb R}^n} e^{-\lVert\cdot\rVert_K}.
\end{align*}
Therefore, 
$$
\left({\mathcal R}_{{\mathcal J}_{\alpha}^l}(\varphi)-\lambda\right)\int_{{\mathbb R}^n} e^{-\varphi}
\leq\left(\frac{1}{n!}-\lambda\right)\int_{{\mathbb R}^n} e^{-\lVert\cdot\rVert_K}\leq 0,
$$
where equality occurs in the first inequality if and only if $\varphi^*$ is linear; in particular, if ${\mathcal R}_{{\mathcal J}_{\alpha}^l}(\varphi)=\lambda=\frac{1}{n!}$
then $\varphi^*$ is linear.

Hence,
${\mathcal R}_{{\mathcal J}_{\alpha}^l}(\varphi)\leq\lambda$ for every $\varphi\in\overline{\text{Cvx}}_{0}\left({\mathbb R}^{n}\right)$, i.e., $\lambda_n(\alpha)\leq\lambda$.
If $\lambda=\frac{1}{n!}$, it follows that $\lambda_n(\alpha)=\frac{1}{n!}$, by \eqref{eq:sl-norm-K}.
Moreover, if ${\mathcal R}_{{\mathcal J}_{\alpha}^l}(\varphi)=\lambda=\frac{1}{n!}$, then 
$\varphi^*$ is linear, i.e., for every $z\in[0,\infty)$ and $t\in[0,1]$ it holds that $\sqrt[n]{{\rm Vol}_n(L_{t z}(\varphi))}=t\sqrt[n]{{\rm Vol}_n(L_z(\varphi))}$ and hence, since $L_{t z}(\varphi)\supseteq t L_z(\varphi)$, by \eqref{eq:concavity-levels}, it follows that $L_{t z}(\varphi)=t L_z(\varphi)$.
Therefore,  $L_z(\varphi)=zL_1(\varphi)$ for every $z\in[0,\infty)$, i.e., $\varphi=\lVert\cdot\rVert_{L_1(\varphi)}$, and $L_1(\varphi)\in{\mathcal K}_0\left({\mathbb R}^n\right)$.
In light of \eqref{eq:sl-norm-K}, this concludes the proof of the proposition. 
\end{proof}

\subsection{Proof of Theorem \ref{thm:exact-Jl}}

\mbox{}\smallskip

Define a function $q:(0,1]\to{\mathbb R}$ as follows;
for every $0<x\leq 1$, let
$$q(x):=\left(1-\sqrt{1-x}\right)x^{-\frac{1}{n+2}}e^{\sqrt{1-x}}.$$ 

\begin{lem}\label{lem:rhon}
The function $q$ is strictly increasing in the interval $(0,1]$; 
consequently, there is a unique $\rho_n\in(0,\frac{1}{4})$ such that
$$
q(4\rho_n)=
\left(\frac{2\sqrt[n]{n!}}{n+2}\right)^{\frac{n}{n+2}},
$$
i.e.,
\begin{equation}\label{eq:rho}
\left(\frac{2}{n+2}\right)^n\bigg/\left(
q\left(4\rho_n\right)\right)^{n+2}=\frac{1}{n!}.
\end{equation}
\end{lem}

\begin{proof}
For every $0<x<1$
\begin{equation*}
q'(x)=\left(\frac{1}{2}-\frac{1}{(n+2)\left(1+\sqrt{1-x}\right)}\right)x^{-\frac{1}{n+2}}e^{\sqrt{1-x}}>\left(\frac{1}{2}-\frac{1}{n+2}\right)x^{-\frac{1}{n+2}}e^{\sqrt{1-x}}>0
\end{equation*}
which implies the first claim; the second claim then follows, since $\sqrt[n]{n!}\leq\frac{1+2+\cdots+n}{n}=\frac{n+1}{2}$ by the AMGM inequality, and hence 
\begin{equation*}\lim_{x\to 0^+}q(x)=0<\left(\frac{2\sqrt[n]{n!}}{n+2}\right)^{\frac{n}{n+2}}<1=q(1).
\qedhere\end{equation*}
\end{proof}

\begin{lem}
\label{lem:h-zeta1}
Let $\alpha\in(0,\frac{1}{4}(n+2)^2)$ and let 
\begin{equation*}
0<\zeta_1(\alpha)=\frac{n+2}{2}\left(1-\sqrt{1-\frac{4\alpha}{(n+2)^2}}\right)
<\zeta_2(\alpha)=\frac{n+2}{2}\left(1+\sqrt{1-\frac{4\alpha}{(n+2)^2}}\right)
\end{equation*}
be the roots of $h_{\alpha}'$ (see \eqref{eq:derivative-h}). 
Then, $h_{\alpha}(\zeta_1(\alpha))\leq\frac{1}{n!}$ if and only if $\alpha\geq\rho_n (n+2)^2$.
\end{lem}

\begin{proof}
Note that
$\frac{\alpha}{\zeta_1(\alpha)}=\zeta_2(\alpha)=n+2-\zeta_1(\alpha)$, therefore
\begin{equation}
\label{eq:by-zeta2}
h_{\alpha}(\zeta_1(\alpha))=\alpha\frac{e^{\zeta_1(\alpha)-\frac{\alpha}{\zeta_1(\alpha)}}}{\zeta_1(\alpha)^{n+2}}=\alpha
\frac{e^{2\zeta_1(\alpha)-(n+2)}}{\zeta_1(\alpha)^{n+2}}
\end{equation}
and hence
\begin{align*}
h_{\alpha}(\zeta_1(\alpha))&=\alpha
\frac{e^{2\zeta_1(\alpha)-(n+2)}}{\zeta_1(\alpha)^{n+2}}=
\left(\frac{2}{n+2}\right)^n\frac{4\alpha}{(n+2)^2}\Bigg/\left(\frac{2\zeta_1(\alpha)}{n+2}e^{1-\frac{2\zeta_1(\alpha)}{n+2}}\right)^{n+2}\\
&=\left(\frac{2}{n+2}\right)^n\Bigg/\left(
q\left(\frac{4\alpha}{(n+2)^2}\right)\right)^{n+2}.
\end{align*}
Hence, since the function 
$q$ is increasing in the interval $(0,1]$, by Lemma \ref{lem:rhon}, it follows that $\alpha\geq\rho_n(n+2)^2$ if and only if
$$
h_{\alpha}(\zeta_1(\alpha))
\leq\left(\frac{2}{n+2}\right)^n\bigg/\left(
q\left(4\rho_n\right)\right)^{n+2}
$$
i.e., $h_{\alpha}(\zeta_1(\alpha))\leq\frac{1}{n!}$, by \eqref{eq:rho}.
\end{proof}

Theorem \ref{thm:exact-Jl} now easily follows.

\begin{proof}[Proof of Theorem \ref{thm:exact-Jl}]
Combining the proof of Lemma \ref{lem:sign-changes-general} with Lemma \ref{lem:h-zeta1} readily imply that for every $\alpha\in[\rho_n(n+2)^2,\infty)$, the function $h_{\alpha}-\frac{1}{n!}$ has at most two positive real roots. 
Proposition \ref{prop:1root} therefore implies the first part of the theorem.
The second part readily follows from the first part, by Observation \ref{obs:J-UpperLowerBound}. 
\end{proof}

\subsection{Proof of Theorem \ref{thm:tight-Jl}}

\begin{lem}\label{lem:Jmaximizers}
For every $\alpha\in(0,\infty)$ there are $r\in[0,1]$ and $t_0\in(0,\infty)$ such that for every  $K\in{\mathcal K}_0\left({\mathbb R}^n\right)$,
$$
\lambda_n(\alpha)= {\mathcal R}_{{\mathcal J}_{\alpha}^l}\left(\psi_{K,r,t_0}\right).
$$
\end{lem}

\begin{proof}
If $\lambda_n(\alpha)=\frac{1}{n!}$, then the statement holds for $r=1$ and any $t_0\in(0,\infty)$, by \eqref{eq:sl-norm-K}.
Hence, by \eqref{eq:sl-norm-K}, we assume that $\lambda_n(\alpha)>\frac{1}{n!}$, and hence, $0<\alpha<\rho_n(n+2)^2$, by Theorem \ref{thm:exact-Jl}.
By Proposition \ref{prop:1root}, the function $h_{\alpha}-\frac{1}{n!}$ has three roots $0<z_1<z_2<z_3$. 

Note that for every $K\in{\mathcal K}_0\left({\mathbb R}^n\right)$, $r\in[0,1]$ and $t_0\in(0,\infty)$ it holds by \eqref{eq:Jint-by-levels}, \eqref{eq:int-by-levels} and \eqref{eq:psi-p}, that
\begin{align*}
{\mathcal R}_{{\mathcal J}_{\alpha}^l}\left(\psi_{K,r,t_0}\right)&=\frac{\int_{{\mathbb R}^n} e^{-{\mathcal J}_{\alpha}^l\psi_{K,r,t_0}}}{\int_{{\mathbb R}^n} e^{-\psi_{K,r,t_0}}}=\frac{\int_0^\infty \alpha e^{-\frac\alpha z}z^{-(n+2)}{\rm Vol}_n\left(L_z(\psi_{K,r,t_0})\right)\,{\rm d}z}{\int_0^\infty e^{-z}{\rm Vol}_n\left(L_z(\psi_{K,r,t_0})\right)\,{\rm d}z}\\
&=\frac{\int_0^\infty \alpha e^{-\frac\alpha z}z^{-(n+2)}\left(\left(\psi_{K,r,t_0}\right)^*(z)\right)^n\,{\rm d}z}{\int_0^\infty e^{-z}\left(\left(\psi_{K,r,t_0}\right)^*(z)\right)^n\,{\rm d}z}\\
&=\frac{\int_0^\infty \alpha e^{-\frac\alpha z}z^{-(n+2)}\left(\sqrt[n]{{\rm Vol}_n(K)}\,p_{r,t_0}(z)\right)^n\,{\rm d}z}{\int_0^\infty  e^{-z}\left(\sqrt[n]{{\rm Vol}_n(K)}\,p_{r,t_0}(z)\right)^n\,{\rm d}z}\\
&=\frac{\int_0^\infty \alpha e^{-\frac\alpha z}z^{-(n+2)}\left(p_{r,t_0}(z)\right)^n\,{\rm d}z}{\int_0^\infty  e^{-z}\left(p_{r,t_0}(z)\right)^n\,{\rm d}z},
\end{align*}
i.e.,
\begin{equation}
\label{eq:sJpsi-sigma}
{\mathcal R}_{{\mathcal J}_{\alpha}^l}\left(\psi_{K,r,t_0}\right)=\sigma_{\alpha}(r,t_0),
\end{equation}
where the function $\sigma_{\alpha}:[0,1]\times(0,\infty)\to(0,\infty)$ is defined by 
$$(r,t_0)\mapsto\frac{\int_0^\infty \alpha e^{-\frac\alpha z}z^{-(n+2)}\left(p_{r,t_0}(z)\right)^n\,{\rm d}z}{\int_0^\infty e^{-z}\left(p_{r,t_0}(z)\right)^n\,{\rm d}z}.$$ 

The function $\sigma_{\alpha}$
is continuous,
hence it attains a maximum $\lambda_0$ in the compact domain $[0,1]\times[z_1,z_2]$.

For every $t_0\in(0,\infty)$ and any $K\in{\mathcal K}_0\left({\mathbb R}^n\right)$ it holds, by \eqref{eq:sl-norm-K}, that 
$$
\sigma_{\alpha}(1,t_0)={\mathcal R}_{{\mathcal J}_{\alpha}^l}\left(\psi_{K,1,t_0}\right)={\mathcal R}_{{\mathcal J}_{\alpha}^l}(\lVert\cdot\rVert_K)=\frac{1}{n!}.
$$
Hence, $\frac{1}{n!}\leq\lambda_0$, and obviously $\lambda_0\leq \lambda_n(\alpha)$.
Assume by contradiction that $\lambda_0<\lambda_n(\alpha)$.
Then, the function $h_{\alpha}-\lambda_0$ has three roots $0<\tilde{z}_1<\tilde{z}_2<\tilde{z}_3$, by Proposition \ref{prop:1root}, and there is $\varphi\in\overline{\text{Cvx}}_{0}\left({\mathbb R}^{n}\right)$ such that ${\mathcal R}_{{\mathcal J}_{\alpha}^l}(\varphi)>\lambda_0$.
Clearly, $z_1\leq\tilde{z}_1<\tilde{z}_2\leq z_2$.
The function $\varphi^*:[0,\infty)\to[0,\infty)$ is a concave increasing function.
By the second part of Lemma \ref{lem:piecewise}, there are 
$r\in[0,1]$ and $t_0\in[\tilde{z}_1,\tilde{z}_2]$
such that
$$
\int_0^{\infty} e^{-z}(h_{\alpha}(z)-\lambda_0)\left(\varphi^*(z)\right)^n  \,{\rm d}z
\leq 
\int_0^\infty e^{-z}(h_{\alpha}(z)-\lambda_0)
\Big(\frac{\varphi^*(\tilde{z}_1)}{\tilde{z}_1}
p_{r,t_0}(z)
\Big)^n
{\rm d}z.
$$
Take $K\in{\mathcal K}_0\left({\mathbb R}^n\right)$ such that $\sqrt[n]{{\rm Vol}_n(K)}=\frac{\varphi^*(z_1)}{z_1}$.
Then, by \eqref{eq:psi-p},
$$
\int_0^{\infty} e^{-z}(h_{\alpha}(z)-\lambda_0)\left(\varphi^*(z)\right)^n  \,{\rm d}z
\leq 
\int_0^\infty e^{-z}(h_{\alpha}(z)-\lambda_0)\left(
\psi_{K,r,t_0}^*(z)\right)^n \,{\rm d}z,
$$
i.e., by \eqref{eq:sl-to-h} and \eqref{eq:sJpsi-sigma}, 
\begin{align*}
0&<\left({\mathcal R}_{{\mathcal J}_{\alpha}^l}(\varphi)-\lambda_0\right)\int_{{\mathbb R}^n} e^{-\varphi}\leq\left({\mathcal R}_{{\mathcal J}_{\alpha}^l}(\psi_{K,r,t_0})-\lambda_0\right)\int_{{\mathbb R}^n} e^{-\psi_{K,r,t_0}}\\
&=\left(\sigma_{\alpha}\left(r,t_0\right)-\lambda_0\right)\int_{{\mathbb R}^n} e^{-\psi_{K,r,t_0}}
\leq 0,
\end{align*}
and we arrived to a contradiction, which concludes the proof.
\end{proof}

\begin{lem}\label{lem:Jmax}
There is universal constant $C>0$ such that for every $\alpha\in(1,\rho_n (n+2)^2)$,  
$$\lambda_n(\alpha)\leq Cn\sqrt{n} e^{2\alpha/n}\frac{n!}{\alpha^n}.$$
\end{lem}

\begin{proof}
By Lemma \ref{lem:h-zeta1},
$h_{\alpha}(\zeta_1(\alpha))>\frac{1}{n!}
$, where
$$
\zeta_1(\alpha)=\frac{n+2}{2}\left(1-\sqrt{1-\frac{4\alpha}{(n+2)^2}}\right)
$$
and the function $h_{\alpha}-h_{\alpha}(\zeta_1(\alpha))$ has exactly two roots, by the proof of Lemma \ref{lem:sign-changes-general}.
Hence, by Proposition \ref{prop:1root} and \eqref{eq:by-zeta2},
\begin{equation*}
\lambda_n(\alpha)\leq h_{\alpha}\left(\zeta_1(\alpha)\right)=\alpha
\frac{e^{2\zeta_1(\alpha)-(n+2)}}{\zeta_1(\alpha)^{n+2}}.
\end{equation*}
Note that
$$
\zeta_1(\alpha)=\frac{2}{1+\sqrt{1-\frac{4\alpha}{(n+2)^2}}}\cdot\frac{\alpha}{n+2},$$
hence,
$$
\frac{\alpha}{n+2}<\zeta_1(\alpha)<\frac{2\alpha}{n+2}<2\rho_n(n+2)<\frac{n+2}{2}.
$$
Therefore, since the function $t\mapsto e^{2t}t^{-(n+2)}$ is decreasing in the interval $(0,\frac{n+2}{2}]$,
\begin{align*}
\lambda_n(\alpha)&\leq\alpha
\frac{e^{2\zeta_1(\alpha)-(n+2)}}{\zeta_1(\alpha)^{n+2}}<\alpha
\frac{e^{\frac{2\alpha}{n+2}-(n+2)}}{(\alpha/(n+2))^{n+2}}=\frac{1}{e^2\alpha}\left(1+\frac{2}{n}\right)^n (n+2)^2 e^{\frac{2\alpha}{n+2}}\left(\frac{n}{e\alpha}\right)^n
\\&<(n+2)^2e^{\frac{2\alpha}{n}}\left(\frac{n}{e\alpha}\right)^n=
\left(1+\frac{2}{n}\right)^2\frac{(n/e)^n\sqrt{n}}{n!}\cdot n\sqrt{n}e^{2\alpha/n}\frac{n!}{\alpha^n}
\end{align*}
and the claim follows since 
$$
\left(1+\frac{2}{n}\right)^2\frac{(n/e)^n\sqrt{n}}{n!}
$$
is bounded above by a universal constant.
\end{proof}

\begin{lem}
\label{lem:lambda-lower}
There is a universal constant $c>0$ such that for every $\alpha\in(1,\infty)$,  
$$\lambda_n(\alpha)\geq\left(1+\frac{c}{n}\right)\frac{n!}{\alpha^n}.$$
\end{lem}

\begin{proof}
It was shown in \cite{FlorentinSegal} that there is a universal constant $c>0$ and $t_0\in(0,\infty)$ such that for any $K\in{\mathcal K}_0\left({\mathbb R}^n\right)$,
$$
\mathcal{R}_{\mathcal J}\left(\max\left\{\|\cdot\|_{K}, 1^{\infty}_{t_0 K} \right\}\right)\geq\left(1+\frac{c}{n}\right)n!.
$$
Note that ${\mathcal J}\left(\max\left\{\|\cdot\|_{K}, 1^{\infty}_{t K} \right\}\right)=\max\left\{\|\cdot\|_{tK}, 1^{\infty}_K \right\}$ for every $t\in(0,\infty)$.
It follows that for every $\alpha\in(1,\infty)$ and any $K\in{\mathcal K}_0\left({\mathbb R}^n\right)$,
\begin{align*}
\lambda_n(\alpha)&\geq{\mathcal R}_{{\mathcal J}_{\alpha}^l}\left(\max\left\{\|\cdot\|_{K}, 1^{\infty}_{\alpha t_0 K} \right\}\right)=\frac{\int_{{\mathbb R}^n} e^{-\alpha\max\left\{\|\cdot\|_{\alpha t_0 K}, 1^{\infty}_K \right\}}}{\int_{{\mathbb R}^n} e^{-\max\left\{\|\cdot\|_{K}, 1^{\infty}_{\alpha t_0 K} \right\}}}\\
&=\frac{\int_K e^{-\|\cdot\|_{t_0 K}}}{\int_{\alpha t_0 K} e^{-\|\cdot\|_K}}=\frac{\int_K e^{-\|\cdot\|_{t_0 K}}}{\alpha^n\int_{t_0 K} e^{-\alpha\|\cdot\|_K}}\geq\frac{\int_K e^{-\|\cdot\|_{t_0 K}}}{\alpha^n\int_{t_0 K} e^{-\|\cdot\|_K}}\\
&=\frac{\int_{{\mathbb R}^n} e^{-\max\left\{\|\cdot\|_{t_0 K}, 1^{\infty}_K \right\}}}{\alpha^n\int_{{\mathbb R}^n} e^{-\max\left\{\|\cdot\|_{K}, 1^{\infty}_{t_0 K} \right\}}}=\frac{1}{\alpha^n}\mathcal{R}_{{\mathcal J}}\left(\max\left\{\|\cdot\|_{K}, 1^{\infty}_{t_0 K} \right\}\right)\geq\left(1+\frac{c}{n}\right)\frac{n!}{\alpha^n}.
\qedhere
\end{align*}
\end{proof}

Theorem \ref{thm:tight-Jl} immediately follows by combining Lemma \ref{lem:Jmaximizers}, Lemma \ref{lem:Jmax}, Lemma \ref{lem:lambda-lower} and Observation \ref{obs:J-UpperLowerBound}.

\section{Scaling the $\mathcal A$ transform}
\label{sec:A}

\begin{prop}
\label{prop:Lc-Santalo-upper-even}
For every $\alpha\in(0,\infty)$ and every even $\varphi\in\overline{\text{Cvx}}_{0}\left({\mathbb R}^{n}\right)$ it holds that
\begin{equation*}
\mathcal{P}_{{\mathcal A}_{\alpha}}\left(\varphi\right)\leq \left({\rm Vol}_n\left(B_{2}^{n}\right)\right)^2n!\left(\frac{1}{\sqrt[n]{n!}}+\frac{\sqrt[n]{n!}}{\alpha}\right)^n.
\end{equation*}
\end{prop}

\begin{proof}
One can repeat the proof of \cite[Proposition 7]{ArtSlom} (or Lehec's argument, using Ball's inequality, which is illustrated in the remark following the proof) pretty much verbatim. 
However, we will use a slightly more straightforward estimate. 

By \eqref{eq:A-level-sets} and Santaló inequality for symmetric convex
bodies, we have that for every $t,s\in(0,\infty)$,
\begin{multline*}
{\rm Vol}_n\left(L_{t}\left(\varphi\right)\right){\rm Vol}_n\left(L_{s}\left({\mathcal A}_{\alpha}\varphi\right)\right)  \le{\rm Vol}_n\left(L_{t}\left(\varphi\right)\right){\rm Vol}_n\left(\left(\frac{ts}{\alpha}+1\right)\left(L_{t}\left(\varphi\right)\right)^{\circ}\right)\\
=\left(\frac{ts}{\alpha}+1\right)^{n}{\rm Vol}_n\left(L_{t}\left(\varphi\right)\right){\rm Vol}_n\left(L_{t}\left(\varphi\right)^{\circ}\right)\le\left(\frac{ts}{\alpha}+1\right)^{n}\left({\rm Vol}_n\left(B_{2}^{n}\right)\right)^{2}.
\end{multline*}
Therefore, by \eqref{eq:int-by-levels}
\begin{align*}
\mathcal{P}_{{\mathcal A}_{\alpha}}\left(\varphi\right)&=\int_{{\mathbb R}^n} e^{-\varphi}\int_{{\mathbb R}^n} e^{-{\mathcal A}_{\alpha}\varphi}=\int_{0}^{\infty}e^{-t}{\rm Vol}_n\left(L_{t}\left(\varphi\right)\right)\,{\rm d}t \int_{0}^{\infty}e^{-s}{\rm Vol}_n\left(L_{s}\left({\mathcal A}_{\alpha}\varphi\right)\right)\,{\rm d}s\\
&=\int_{0}^{\infty}\int_{0}^{\infty}e^{-t}e^{-s}{\rm Vol}_n\left(L_{t}\left(\varphi\right)\right){\rm Vol}_n\left(L_{s}\left({\mathcal A}_{\alpha}\varphi\right)\right)\,{\rm d}t \,{\rm d}s\\
&\leq\int_{0}^{\infty}\int_{0}^{\infty}e^{-t}e^{-s}\left(\frac{ts}{\alpha}+1\right)^{n}\left({\rm Vol}_n\left(B_{2}^{n}\right)\right)^{2}\,{\rm d}t \,{\rm d}s.
\end{align*}
The claim follows since, by using that the sequence $(k!^{1/k})_{k=1}^{\infty}$ is increasing,
\begin{align*}
\int_{0}^{\infty}\int_{0}^{\infty}&e^{-t}e^{-s}\left(\frac{ts}{\alpha}+1\right)^{n}\,{\rm d}t \,{\rm d}s= \int_{0}^{\infty}\int_{0}^{\infty}e^{-t}e^{-s}\sum_{k=0}^{n}\binom{n}{k}\frac{t^k s^k}{\alpha^k}\,{\rm d}t \,{\rm d}s\\
& =\sum_{k=0}^{n}{\binom{n}{k}}\frac{1}{\alpha^k}\int_{0}^{\infty}e^{-t}t^k \,{\rm d}t\int_{0}^{\infty}e^{-s}s^k \,{\rm d}s =\sum_{k=0}^{n}{\binom{n}{k}}\frac{k!^2}{\alpha^k}\leq\sum_{k=0}^{n}{\binom{n}{k}}\frac{\left(n!^{k/n}\right)^2}{\alpha^k}\\&= \sum_{k=0}^{n}{\binom{n}{k}}\frac{\left(n!^{2/n}\right)^k}{\alpha^k}=\left(1+\frac{n!^{2/n}}{\alpha}\right)^n=n!\left(\frac{1}{\sqrt[n]{n!}}+\frac{\sqrt[n]{n!}}{\alpha}\right)^n.
\qedhere
\end{align*}
\end{proof}

\begin{proof}[Proof of Theorem \ref{thm:Lc-Santalo}]
Note that
for every $\alpha\in(0,\infty)$ and 
$\varphi\in\overline{\text{Cvx}}_{0}\left({\mathbb R}^{n}\right)$,  
\begin{equation}\label{eq:sAsJsL}
\mathcal{P}_{{\mathcal A}_{\alpha}}\left(\varphi\right)={\mathcal R}_{{\mathcal J}_{\alpha}^l}({\mathcal L}\varphi)\left(\int_{{\mathbb R}^n} e^{-{\mathcal L}\varphi}\int_{{\mathbb R}^n} e^{-\varphi}\right)
\end{equation}
(similar observation was already made in \cite[\S 2]{FlorentinSegal}).
It follows, by \eqref{eq-ML}, that for every $\alpha\in(0,\infty)$,
\begin{equation*}
\inf_{\substack{\varphi\in\overline{\text{Cvx}}_{0}\left({\mathbb R}^{n}\right),\\{\rm bar}\left(e^{-\varphi}\right)=0}}\mathcal{P}_{{\mathcal A}_{\alpha}}\left(\varphi\right)\geq\inf_{\varphi\in\overline{\text{Cvx}}_{0}\left({\mathbb R}^{n}\right)}\mathcal{P}_{{\mathcal A}_{\alpha}}\left(\varphi\right)\geq c_{\mathcal L}^n\inf_{\varphi\in\overline{\text{Cvx}}_{0}\left({\mathbb R}^{n}\right)}{\mathcal R}_{{\mathcal J}_{\alpha}^l}(\varphi),
\end{equation*}
and the lower bounds in \eqref{eq:Lc-Santalo-large-alpha} and \eqref{eq:Lc-Santalo-small-alpha} immediately follow from the second part of Theorem \ref{thm:exact-Jl} and the lower bound in Theorem \ref{thm:tight-Jl}, respectively.

For every $\varphi\in\overline{\text{Cvx}}_{0}\left({\mathbb R}^{n}\right)$ such that ${\rm bar}\left(e^{-\varphi}\right)=0$, it holds, by \cite[Theorem 1.3]{AKM}, that
\begin{equation*}
\int_{{\mathbb R}^n} e^{-{\mathcal L}\varphi}\int_{{\mathbb R}^n} e^{-\varphi}\leq(2\pi)^n
\end{equation*}
and it follows, by \eqref{eq:sAsJsL}, that for every $\alpha\in(0,\infty)$, 
\begin{equation*}
\sup_{\substack{\varphi\in\overline{\text{Cvx}}_{0}\left({\mathbb R}^{n}\right),\\{\rm bar}\left(e^{-\varphi}\right)=0}}\mathcal{P}_{{\mathcal A}_{\alpha}}\left(\varphi\right)\leq(2\pi)^n\sup_{\varphi\in\overline{\text{Cvx}}_{0}\left({\mathbb R}^{n}\right)}{\mathcal R}_{{\mathcal J}_{\alpha}^l}(\varphi).
\end{equation*}
The upper bounds in \eqref{eq:Lc-Santalo-large-alpha} and \eqref{eq:Lc-Santalo-small-alpha} immediately follow from the first part of Theorem \ref{thm:exact-Jl} and the upper bound in Theorem \ref{thm:tight-Jl}, respectively.

Proposition \ref{prop:Lc-Santalo-upper-even} completes the proof of the Theorem.
\end{proof}

\section{A  K\"{o}nig-Milman duality of entropy result for $\mathcal A$}
\label{sec:MP-KM}

The proof of Theorem \ref{thm:KM_polarity}  mainly relies  on  Theorem \ref{thm:Lc-Santalo} but requires some additional preparation. 

Recall the operation of inf-convolution of two  convex functions $\varphi,\psi\in{\rm Cvx}({\mathbb R}^n)$, defined by 
$$
(\varphi\square\psi)\left(x\right):=\inf_{y\in{\mathbb R}^n}\left(\varphi\left(y\right)+\psi\left(x-y\right)\right).
$$
The g-inf-convolution of two geometric convex functions $\varphi,\psi\in{\rm Cvx}_0({\mathbb R}^n)$ is given by 
$$
\varphi\boxdot\psi:={\mathcal J}\left({\mathcal J}\varphi\,\square\,{\mathcal J}\psi\right).
$$
It is known that the polarity transform exchanges the operation of g-inf-convolution with  usual summation, that is 
${\mathcal A} (\varphi\boxdot \psi)={\mathcal A}\varphi+{\mathcal A}\psi$, see e.g., \cite{AM-Hidden}. Hence, also 
\begin{equation}
\label{eq:boxdot-A-sum}
 {\mathcal A}_{\alpha} (\varphi\boxdot \psi)={\mathcal A}_{\alpha}\varphi+ {\mathcal A}_{\alpha}\psi,   
\end{equation}
which we will use in the sequel. Since ${\mathcal A}_{\alpha}$ is an order reversing involution on ${\rm Cvx}_0({\mathbb R}^n)$, it follows that $\varphi\boxdot\psi\le\varphi$, which will also be useful.

Rogers and Shephard proved in \cite{rogers1958some} that ${\rm Vol}_n(K-K)/{\rm Vol}_n(K)\leq\binom{2n}{n}$ for every convex body $K\subseteq{\mathbb R}^n$.
The following proposition is a qualitative Rogers-Shephard type inequality for geometric log-concave functions. 

\begin{prop}\label{prop:f_ginf}
There is a universal constant $C$ such that for every $\varphi\in\overline{\text{Cvx}}_{0}\left({\mathbb R}^{n}\right)$,  
\[
\frac{\int_{{\mathbb R}^n} e^{-\varphi\boxdot\varphi^-}}{\int_{{\mathbb R}^n} e^{-\varphi}}\le 8^{n},
\]
where $\varphi^-\in\overline{\text{Cvx}}_{0}\left({\mathbb R}^{n}\right)$ is defined by $\varphi^-(x):=\varphi(-x)$ for every $x\in{\mathbb R}^n$.
\end{prop}
Functional inequalities in this spirit, as well as their optimal bounds,  have been investigated in \cite{AAGJ19}, \cite{AloGonJimVil16} and  \cite{Colesanti05}. In our case, Proposition \ref{prop:f_ginf} is a consequence of \cite[Theorem 2.2]{AloGonJimVil16} and the next lemma. It would be interesting to find the optimal constant $0<c\le8$ for which the proposition holds. 

\begin{lem}
\label{lem:inf_vs_ginf} If  $\varphi,\psi\in{\rm Cvx}_0({\mathbb R}^n)$ then for all $x\in{\mathbb R}^{n}$ we have
\[
2\left(\varphi\boxdot\psi\right)\left(x/2\right)\le\left(\varphi\square\psi\right)\left(x\right)\le2\left(\varphi\boxdot\psi\right)\left(x\right).
\]
\end{lem}

\begin{proof}
By \cite[Proposition 2.2]{AFS20}, we have for all $x\in{\mathbb R}^n$, 
\begin{equation*}
(\varphi\boxdot\psi)\left(x\right)=\inf_{\substack{0<t<1,\,y,z\in{\mathbb R}^n\\\left(1-t\right)y+tz=x}}
\max\left\{ \left(1-t\right)\varphi\left(y\right),t\psi\left(z\right)\right\}\le\inf_{y\in{\mathbb R}^n}\max\left\{ \frac{1}{2}\varphi\left(y\right),\frac{1}{2}\psi\left(2x-y\right)\right\}.
\end{equation*}
Therefore, for every $x\in{\mathbb R}^n$,
\begin{equation*}
\left(\varphi\,\square\,\psi\right)\left(x\right)=\inf_{y\in{\mathbb R}^n}\left(\varphi\left(y\right)+\psi\left(x-y\right)\right)\ge\inf_{y\in{\mathbb R}^n}\max\left\{ \varphi\left(y\right),\psi\left(x-y\right)\right\}
\ge 2\left(\varphi\boxdot\psi\right)\left(x/2\right).
\end{equation*}
By applying the above inequality to the functions ${\mathcal J}\varphi$
and ${\mathcal J}\psi$ we obtain that for every $x\in{\mathbb R}^n$
\[
\left({\mathcal J}\varphi\,\square\,{\mathcal J}\psi\right)\left(x\right)\ge 2\left({\mathcal J}\varphi\,\boxdot\,{\mathcal J}\psi\right)\left(x/2\right),
\]
hence, by \eqref{eq:Jscaling},
$$
\left({\mathcal J}\varphi\,\square\,{\mathcal J}\psi\right)\left(x\right)\ge \left({\mathcal J}\frac{1}{2}{\mathcal J}\left({\mathcal J}\varphi\,\boxdot\,{\mathcal J}\psi\right)\right)\left(x\right)=\left({\mathcal J}\frac{1}{2}\left(\varphi\,\square\,\psi\right)\right)\left(x\right).
$$
Therefore, since $\mathcal J$ is an order preserving involution, it holds for every $x\in{\mathbb R}^n$ that, 
\begin{equation*}
\left(\varphi\boxdot\psi\right)\left(x\right)=\left({\mathcal J}\left({\mathcal J}\varphi\,\square\,{\mathcal J}\psi\right)\right)\left(x\right)\ge\frac{1}{2}\left(\varphi\,\square\,\psi\right)\left(x\right).
\qedhere
\end{equation*}
\end{proof}

\begin{proof}[Proof of Proposition \ref{prop:f_ginf}]
Note that for every function $\phi\in\overline{\text{Cvx}}_{0}\left({\mathbb R}^{n}\right)$ one has 
\begin{equation}
\int_{{\mathbb R}^n} e^{-2\phi}\le\int_{{\mathbb R}^n} e^{-\phi}\le \int_{{\mathbb R}^n} e^{-2\phi(\cdot/2)} =2^{n}\int_{{\mathbb R}^n} e^{-2\phi}.\label{eq:func_KM_square}
\end{equation}
Indeed, the left hand side inequality follows from the fact that $\phi\ge0$
and the right hand side inequality follows from the fact that $\phi(x)\ge 2\phi(x/2)$ for any $x\in{\mathbb R}^n$, which holds  since $\phi$ is convex and $\phi(0)=0$. 
By \cite[Theorem 2.2]{AloGonJimVil16}, we have 
$$
\frac{\int_{{\mathbb R}^n} e^{-\varphi\square\varphi^-}}{\int_{{\mathbb R}^n} e^{-\varphi}}\le \binom{2n}{n}\le 4^n
$$
and hence, by Lemma \ref{lem:inf_vs_ginf} and \eqref{eq:func_KM_square}, it follows that 
\begin{equation*}
\frac{\int_{{\mathbb R}^n} e^{-\varphi\boxdot\varphi^-}}{\int_{{\mathbb R}^n} e^{-\varphi}}\le\frac{\int_{{\mathbb R}^n} e^{-\frac{1}{2}\varphi\square\varphi^-}}{\int_{{\mathbb R}^n} e^{-\varphi}}\le 2^n\frac{\int_{{\mathbb R}^n} e^{-\varphi\square\varphi^-}}{\int_{{\mathbb R}^n} e^{-\varphi}}\le 8^n. \qedhere
\end{equation*}
\end{proof}

We also need the following estimates for covering numbers of geometric log-concave functions from 
\cite{ArtSlom17-func_cov}, which hold true for any $\varphi,\psi\in\overline{\text{Cvx}}_{0}\left({\mathbb R}^{n}\right)$: 
\begin{equation}\label{eq:cov_vol_bd_even}
\frac{\int_{{\mathbb R}^n} e^{-2\varphi}}{\int_{{\mathbb R}^n} e^{-(\varphi+\psi)}}\le N\left(e^{-\varphi},e^{-\psi}\right)\le2^{n}\frac{\int_{{\mathbb R}^n} e^{-2\varphi}}{\int_{{\mathbb R}^n} e^{-(\varphi+\psi)}},
\end{equation}
\begin{equation}\label{eq:cov_vol_bd_infc}
N\left(e^{-\varphi},e^{-\psi}\right)\le\frac{\int_{{\mathbb R}^n} e^{-\varphi\square\psi^-}}{\int_{{\mathbb R}^n} e^{-2\psi^-}}.
\end{equation}
Combined with Lemma \ref{lem:inf_vs_ginf} and \eqref{eq:func_KM_square}, we also have: 
\begin{equation}\label{eq:cov_vol_bd_ginf}
N\left(e^{-\varphi},e^{-\psi}\right)\le\frac{\int_{{\mathbb R}^n} e^{-2(\varphi\boxdot\psi^-)(\cdot/2)}}{\int_{{\mathbb R}^n} e^{-2\psi^-}} =2^n\frac{\int_{{\mathbb R}^n} e^{-\varphi\boxdot\psi^-}}{\int_{{\mathbb R}^n} e^{-2\psi^-}}.
\end{equation}

\begin{proof}[Proof of Theorem \ref{thm:KM_polarity}]
Assume first that $\varphi$ and $\psi$ are both even functions. By (\ref{eq:cov_vol_bd_ginf}) and \eqref{eq:func_KM_square} we see that 
\[
N(e^{-{\mathcal A}_{\alpha}\varphi},e^{-{\mathcal A}_{\alpha}\psi})\le\frac{\int_{{\mathbb R}^n} e^{-2\left(\mathcal{A}_{\alpha}\varphi\boxdot\mathcal{A}_{\alpha}\psi\right)\left(x/2\right)}\,{\rm d}x}{\int_{{\mathbb R}^n} e^{-2{\mathcal A}_{\alpha}\psi}}\le4^{n}\frac{\int_{{\mathbb R}^n} e^{-\left(\mathcal{A}_{\alpha}\varphi\boxdot\mathcal{A}_{\alpha}\psi\right)}}{\int_{{\mathbb R}^n} e^{-{\mathcal A}_{\alpha}\psi}}.
\]
As ${\mathcal A}_{\alpha}({\mathcal A}_{\alpha}\varphi\boxdot{\mathcal A}_{\alpha}\psi)=\varphi+\psi$ and the barycenters of $e^{-\psi}$ and $e^{-(\varphi+\psi)}$ are at the origin, it follows by 
Theorem \ref{thm:Lc-Santalo} that 
\[
 N(e^{-{\mathcal A}_{\alpha}\varphi},e^{-{\mathcal A}_{\alpha}\psi})\le C^n \frac{\int_{{\mathbb R}^n} e^{-\psi}}{\int_{{\mathbb R}^n} e^{-(\varphi+\psi)}}
\]
for some universal constant $C>0$. By (\ref{eq:func_KM_square}) and
(\ref{eq:cov_vol_bd_even}), we obtain 
\begin{equation*}
N(e^{-{\mathcal A}_{\alpha}\varphi},e^{-{\mathcal A}_{\alpha}\psi})\le C^{n} \frac{\int_{{\mathbb R}^n} e^{-\psi}}{\int_{{\mathbb R}^n} e^{-(\varphi+\psi)}}\le (2C)^n \frac{\int_{{\mathbb R}^n} e^{-2\psi}}{\int_{{\mathbb R}^n} e^{-(\varphi+\psi)}}\le (2C)^n N(e^{-\psi}, e^{-\varphi}).
\end{equation*}
To obtain the opposite inequality, one simply replaces 
$\varphi$ with ${\mathcal A}_{\alpha}\psi$ and $\psi$ with ${\mathcal A}_{\alpha}\varphi$, which  completes our proof in the case that $\varphi$ and $\psi$ are even functions. 

Next, we  prove the general case where  $\varphi$ and $\psi$ are not assumed to be
even. Using the relations ${\mathcal A}_{\alpha}(\varphi\boxdot\varphi^-)={\mathcal A}_{\alpha}\varphi+({\mathcal A}_{\alpha}\varphi)^-$ and ${\mathcal A}_{\alpha}(\psi+\psi^-)={\mathcal A}_{\alpha}\psi \boxdot({\mathcal A}_{\alpha}\psi)^-$, it is clearly enough to prove that 
\begin{equation}\label{eq:KM_even_reduction}
N(e^{-\varphi},e^{-\psi})\le N(e^{-\varphi\boxdot\varphi^-},e^{-(\psi+\psi^-)})\le C^n N(e^{-\varphi},e^{-\psi})
\end{equation}
for some universal $C>0$, as this reduces the problem to the case where both functions are even, which is resolved. 

To see the first inequality in \eqref{eq:KM_even_reduction}, note that  $\varphi\ge\varphi\boxdot\varphi^-$ and
$\psi+\psi^-\ge \psi$, and use the  monotonicity property of covering numbers \cite[Fact 2.6]{ArtSlom17-func_cov}.

To see the second inequality in \eqref{eq:KM_even_reduction},  we first use the sub-multiplicativity property of covering numbers \cite[Fact 2.10]{ArtSlom17-func_cov} to get
$$
N(e^{-\varphi\boxdot\varphi^-},e^{-(\psi+\psi^-)})\le N(e^{-\varphi\boxdot\varphi^-}, e^{-\varphi})\cdot N(e^{-\varphi}, e^{-\psi})\cdot N(e^{-\psi},e^{-(\psi+\psi^-)}).
$$
To bound the term $N(e^{-\varphi\boxdot\varphi^-}, e^{-\varphi})$,  we use \eqref{eq:cov_vol_bd_even}, the fact that $\varphi+(\varphi\boxdot\varphi^-)\le2\varphi$
and \eqref{eq:func_KM_square} to get 
\begin{equation*}
N(e^{-\varphi\boxdot\varphi^-}, e^{-\varphi})\le 2^n\frac{\int_{{\mathbb R}^n}{e^{-2(\varphi\boxdot\varphi^-)}}}{\int_{{\mathbb R}^n}{e^{-(\varphi+\varphi\boxdot\varphi^-)}}}\le
2^n\frac{\int_{{\mathbb R}^n}{e^{-2(\varphi\boxdot\varphi^-)}}}{\int_{{\mathbb R}^n}{e^{-2\varphi}}}
\le 8^n\frac{\int_{{\mathbb R}^n}{e^{-\varphi\boxdot\varphi^-}}}{\int_{{\mathbb R}^n}{e^{-\varphi}}}.
\end{equation*}
By Proposition \ref{prop:f_ginf}, it follows that $N(e^{-\varphi\boxdot\varphi^-}, e^{-\varphi})\le 8^{2n}$. 
To bound the term $N(e^{-\psi},e^{-(\psi+\psi^-)})$, we use \eqref{eq:cov_vol_bd_infc},  \eqref{eq:func_KM_square} to get
\begin{equation*}
N(e^{-\psi},e^{-(\psi+\psi^-)})\le
2^n\frac{\int_{{\mathbb R}^n}{e^{-2\psi}}}{\int_{{\mathbb R}^n}{e^{-(2\psi+\psi^-)}}}\le
2^n\frac{\int_{{\mathbb R}^n}{e^{-2\psi}}}{\int_{{\mathbb R}^n}{e^{-2(\psi+\psi^-)}}}\le 8^n\frac{\int_{{\mathbb R}^n}{e^{-\psi}}}{\int_{{\mathbb R}^n}{e^{-(\psi+\psi^-)}}}.
\end{equation*}
By  Theorem \ref{thm:Lc-Santalo} and Proposition \ref{prop:f_ginf}, it follows that for some  universal  $C_1>0$, 
$$
N(e^{-\psi},e^{-(\psi+\psi^-)})\le C_1^n \frac{\int_{{\mathbb R}^n} e^{-(({\mathcal A}_\alpha\psi)\boxdot(A_\alpha\psi)^-)}}{\int_{{\mathbb R}^n} e^{-{\mathcal A}_\alpha\psi}}
\le (8C_1)^n.
$$
Concluding the above, we obtain that 
$$
N(e^{-\varphi\boxdot\varphi^-},e^{-(\psi+\psi^-)})\le (8^{3}C_1)^n\cdot N(e^{-\varphi},e^{-\psi}),
$$
as needed. 
\end{proof}

\bibliographystyle{amsplain}
\bibliography{references}

\vspace*{1cm}

\end{document}